\documentclass[12pt]{article}
\usepackage{amssymb,amsfonts,amsthm,mathtools,float,color,bm}
\usepackage{caption,subcaption,tikz,gensymb,tikz-3dplot}

\usetikzlibrary{positioning}
\usetikzlibrary{shapes}
\tikzset{main node/.style={ellipse,draw,minimum size=1cm,inner sep=1pt},
            }
            \tikzstyle{every node}=[draw, ellipse, minimum size=1cm,align=center]

\usepackage[hmargin=1in,vmargin=1in]{geometry}

\DeclareMathOperator{\lb}{lb}
\DeclareMathOperator{\ls}{ls}
\DeclareMathOperator{\tr}{tr}
\DeclareMathOperator{\rb}{rb}
\DeclareMathOperator{\rs}{rs}
\DeclareMathOperator{\LB}{LB}

\DeclareMathOperator{\LS}{LS}
\DeclareMathOperator{\RB}{RB}

      \theoremstyle{plain}
      \newtheorem{theorem}{Theorem}[section]
      \newtheorem{lemma}[theorem]{Lemma}
      \newtheorem{corollary}[theorem]{Corollary}
      \newtheorem{question}[theorem]{Question}
      \newtheorem{proposition}[theorem]{Proposition}
      \theoremstyle{definition}
      
      \theoremstyle{remark}
      
      \theoremstyle{plain}
      \newtheorem{conjecture}{Conjecture}[section]

\newcommand{\tbl}[1]{\textcolor{blue}{}}
\newcommand{\tgr}[1]{\textcolor{green}{}}

\DeclareMathOperator{\Av}{Av}

\def\multiset#1#2{\ensuremath{\left(\kern-.3em\left(\genfrac{}{}{0pt}{}{#1}{#2}\right)\kern-.3em\right)}}

\newcommand{\gauss}[2]{\genfrac{[}{]}{0pt}{}{#1}{#2}}

\newcommand{\N}{\mathbb{N}}
\newcommand{\C}{\mathbb{C}}
\newcommand{\til}{\widetilde}
\newcommand{\frakS}{\mathfrak{S}}

\begin{document}
\author{Robert Dorward\\
\small Department of Mathematics, Oberlin College\footnote{The author now works at Google.},\\[-5pt]
\small Oberlin, OH 44074, USA, {\tt bobbydorward@gmail.com}}

\title{Combinatorial statistics on restricted growth functions containing a pattern exactly $k$ times\footnote{A thesis completed for the honors program in mathematics at Oberlin College. Completed and published internally at Oberlin in 2017.}}
\maketitle

\tableofcontents

\pagebreak
\section{Introduction}

Enumerative combinatorics concerns itself with counting the number of certain combinatorial objects. Bijection to a set of objects whose cardinality is known is a common way to accomplish this goal. It has the added benefit of uncovering connections between previously unrelated objects. 

In this paper, we begin with the field of permutation patterns and answer some related questions about set partitions. In doing so, we find connections to several well-known combinatorial objects, including integer partitions, generating functions, $q$-analogs, and partially ordered sets.

In particular, in Section~\ref{setparts} we focus on pattern avoidance defined for restricted growth functions which are essentially set partitions in disguise. We characterize the cardinalities of two pattern containment classes of restricted growth functions. In Section~\ref{intpartsec}, we examine a few of the generating functions for these pattern containment classes, finding new connections to integer partitions that expand on existing connections in the literature. In Section~\ref{unimsec} we discuss some properties of the generating functions studied including unimodality and symmetry and give a conjecture on the unimodality of one of the generating functions. To do so, we introduce the theory of posets, formulate the conjecture in this framework and give an overview of the proof of a related problem by Proctor, using linear algebra. All results are new unless otherwise stated. We end with a list of open problems and future work.
\subsection{A Brief History of Permutation Patterns}
The field of permutation patterns as a subfield of combinatorics has experienced an explosion of growth over the past several decades, though its roots can be traced to questions in theoretical computer science. A \emph{permutation} of length $n$ is an ordering of the numbers in $[n]=\{1,2,\dots, n\}$, for example 1423 and 53241. The set of permutations of length $n$ is $\frakS_n$, the symmetric group. In 1968, Donald Knuth \cite{knu:art} gave a characterization of the \emph{stack-sortable} permutations. These are permutations (thought of in single-line notation) which can be sorted back into increasing order using a single stack data structure. Knuth found that what mattered about the permutation was the relative ordering of its elements: permutations that were stack-sortable were exactly those which had no three elements in the order \emph{second largest}, \emph{largest}, \emph{smallest}. 

In the language of permutation patterns, we say that a permutation $\sigma$ \emph{contains} another permutation $\pi$ if there is a subpermutation $\sigma'$ of $\sigma$ whose elements are in the same relative order as $\pi$. A permutation $\sigma$ \emph{avoids} $\pi$ if it does not contain $\pi$. For example $\sigma = 15324$ contains $\pi = 1423$ because the subpermutation $\sigma' = 1524$ has elements in the same relative order as $\pi$. We let $\Av_n(\pi) = \{ \sigma \in \frakS_n : \sigma \text{ avoids } \pi\}$

Thus, Knuth showed that the stack-sortable permutations are precisely those which avoid the pattern 231. What is more interesting from a mathematical viewpoint is that he also found that the stack-sortable permutations (and therefore 231-avoiding permutations) are enumerated by the Catalan numbers, a sequence of integers ubiquitous in combinatorics. 

As this idea of pattern avoidance began to be explored, a remarkable discovery was made. The number of permutations avoiding \emph{any} fixed length three permutation is the $n$th Catalan number (i.e., $\#\Av_n(\pi) = C_n$ for any $\pi\in \frakS_n$), even though the set of permutations is different for each fixed length three permutation.

As the subject expanded, different questions about permutation patterns were considered. These include enumerative questions such as finding exact formulas for the cardinalities of different avoidance classes, characterizing the permutation of length $n$ that maximizes the number of copies of the pattern contained, or finding asymptotic formulas for the cardinalities of avoidance classes. The most famous result in the area was proved in 2004 after being open for roughly 25 years. Originally known as the Stanley-Wilf Conjecture, the Marcus-Tardos Theorem states that the growth rate of any avoidance class of permutations is at most exponential in $n$ \cite{mt:stw}. 

After studying permutations, this idea of pattern avoidance was generalized to many different combinatorial structures such as graphs, words, matrices, tableaux and set partitions. In this paper, we focus on avoidance in a particular set of words well known to be in bijection with set partitions.

\section{Set Partitions, Restricted Growth Functions, and Patterns}\label{setparts}
We begin this section by introducing basic concepts and definitions. Then, we will characterize two containment classes and compute their cardinality.

We use $[n]$ to denote the set $\{1,2,\dots, n\}$. 
A \emph{set partition} is a collection of non-empty subsets $B_1,\dots, B_m \subseteq [n]$ such that 
\begin{align*}
1. \hspace{5pt}& \forall i\neq j, B_i \cap B_j = \emptyset,\\
2. \hspace{5pt}& \bigcup_ {i=1}^n B_i  = [n].
\end{align*}
For notational convenience we will adopt the convention of writing our set partition in \emph{block form} $B_1/ B_2/ \dots / B_m$ and will refer to the subsets $B_i$ as \emph{blocks}. We drop the set braces around each subset as well. We let $\Pi_n$ be the set of partitions of $[n]$.

A set partition is said to be in \emph{standard form} if $\min B_1 < \min B_2 < \dots <\min B_m$ and the elements of each $B_i$ are written in ascending order. For example, the set partitions $1/25/34$ and $13/2/467/58$ are set partitions written in standard form.  Standard form is relevant to introduce the main combinatorial object studied in this paper.

A \emph{restricted growth function} (RGF) is a sequence of positive integers $w = w_1\dots w_n$ such that 
\begin{align*}
1. \hspace{5pt}&w_1 = 1,\\
2.  \hspace{5pt}&w_i \leq 1 + \max\{w_1,\dots, w_{i-1}\},
\end{align*}
for $2\leq i \leq n$. RGFs are of interest because they are in bijection with set partitions via the following bijection: Given a set partition $B_1/\dots/B_m\vdash [n]$ in standard form, create the RGF $w_1\dots w_n$ where $w_i = j$ if $i \in B_j$. The set partitions above correspond to RGFs $12332$ and $12134334$. Because our set partition is in standard form, we have $w_1=1$ and the growth condition holds. We let $R_n$ be the set of restricted growth functions of length $n$.

We say that an RGF $w$ \emph{contains} another RGF $v$ if there is a subword $w'$ of $w$ such that the elements of $w$ are in the same relative order as $v$. We refer to $v$ as the \emph{pattern}. For example $w= 12231245$ contains the pattern $v= 11123$ because the elements of the subword $w'=22245$ of $w$ are in the same relative order as $11123$. We let $R_{n,k}(v)$ be the set of restricted growth functions of length $n$ that contain exactly $k$ copies of $v$. Continuing our example, we have $1221245 \in R_{7,1}(11123)$.

Historically, the case most frequently considered is $k=0$. If $w$ does not contain $v$ we say that $w$ \emph{avoids} $v$. Sagan \cite{sag:pas} determined $\#R_{n,0}(v)$ for all $v$ of length 3 and for the most part found that they have simple and elegant formulas. For example, $$\#R_{n,0}(121) = 2^{n-1}$$ and $$\#R_{n,0}(111) = \sum_{i\geq 0} \binom{n}{2i}(2i)!!,$$ where $$(2i)!! = (2i-1)(2i-3)(2i-5)\cdots (3)(1).$$

Two patterns $w$ and $v$ are called \emph{Wilf equivalent} if $\#R_{n,0}(w)=\#R_{n,0}(v)$ for all $n\geq0$. Sagan showed that the patterns 112, 122, 121, and 123 are all Wilf equivalent. This motivates a more general definition of Wilf equivalence: for a fixed $k\geq 0$, we say that $w$ and $v$ are $k$-Wilf equivalent if $\#R_{n,k}(w) = \#R_{n,k}(v)$ for all $n\geq 0$. In \cite{man:set} (Research Direction 6.6) Mansour gives a list of suspected $1$-Wilf equivalences for patterns of length at most 5, though offers no proof for them. We will prove one of these $1$-Wilf equivalences and then show that in fact, a stronger relationship holds.

\subsection{Characterizations of $R_{n,1}(112)$ and $R_{n,1}(122)$}
We will focus on $R_{n,1}(112)$ and $R_{n,1}(122)$ as we will see later that they have interesting properties, first giving characterizations of their elements. In doing so, we also show that 
$$\#R_{n,1}(112) = \#R_{n,1}(122) = (n-2)2^{n-3},$$
which were stated in \cite{man:set} without proof.

Note that by definition it must be the case that $\#R_{n,1}(v) =0$ when $n$ is less than the length of $v$. Thus, in this section, all formulas will hold only for $n\geq 3$.

We begin by characterizing $R_{n,1}(112)$. In order to achieve this, we will need to know the characterization of $R_{n,0}(112)$ given by Sagan, stated here without proof.
\begin{lemma}[\cite{sag:pas}]\label{sagan112}
We can characterize the elements of $R_{n,0}(112)$ as 
$$123\dots m w',$$
where $w'$ is weakly decreasing. In addition we have that 
$$\# R_{n,0}(112) = 2^{n-1}.$$
\end{lemma}
%\begin{proof}\tre{need to decide whether or not to include this}
%\end{proof}

Moving from $k=0$ to $k=1$, we find that a similar characterization holds. We say that a word $w$ contains an \emph{noninversion} at index $i$ if there exists $j>i$ such that $w_i < w_j$. In addition we will often refer to the \emph{initial run} of a restricted growth function $w$, which is the longest initial strictly increasing subword of $w$. For example, the initial run of $1234555123$ is $12345$.
\begin{lemma}\label{lem112card}
We can characterize the elements of $R_{n,1}(112)$ as
$$123\dots (m-1) w',$$
where $w'$ contains the maximum, $m$, and is a sequence with exactly one noninversion.
\end{lemma}
\begin{proof}
Let $w\in R_{n,1}(112)$ and let $abc$ be the subword of $w$ that is the copy of $112$ in $w$. Let $i$ be the index of $b$ and let $m$ be the maximum of $w$. It must be the case that $i>(m-1)$, or $w$ would contain more than one copy of 112. This is because $i\leq (m-1)$ implies that the elements $(m-1)$ and $m$ appear after $b$, which creates the copies of 112, $ab(m-1)$ and $abm$. This also implies that $w$ contains an initial run up at least up to $(m-1)$ (and possibly up to $m$). 

Now suppose, seeking a contradiction, that there were two noninversions in which both elements of each noninversion appeared after the initial run. Label them $a_1b_1$ and $a_2b_2$. Because we fixed the max of $w$ to be $m$ and $a_1,a_2< m$, it must be the case that another copy of $a_1$ and $a_2$  (which we will label $c_1$ and $c_2$, respectively) appear in the initial run of $w$. However, this creates two copies of 112: $c_1a_1b_1$ and $c_2a_2b_2$, which implies that $w\not \in R_{n,1}(112)$. Thus, we choose our one noninversion and then fill the rest of the word in with a weakly decreasing sequence.
\end{proof}
Now that we have a characterization of $R_{n,1}(112)$ we can determine its cardinality.
\begin{theorem}\label{112card}
We have
$$\#R_{n,1}(112) = (n-2)2^{n-3}.$$
\end{theorem}
\begin{proof}
Let $w\in R_{n,1}(112)$ and let $abc$ be the copy of $112$ in $w$. By Lemma \ref{lem112card}, we know that if we delete $b$ from $w$ then we get a word $w'\in R_{n-1,0}(112)$. Thus we can obtain $\#R_{n,1}(112)$ by counting how many different ways we can insert some $b$ into a given $w'\in R_{n-1,0}(112)$. 

To do this, fix a word $w'\in R_{n-1,0}(112)$ and let $m$ be the maximum of $w'$. Next, let $b\in\{1,\dots,m-1\}$. By \ref{lem112card} we know that $b$ must be placed after the initial run $12\dots (m-1)$ and in order to produce exactly one copy of $112$, it must be placed before exactly one larger element. Let $c' = \min\{w_i: i>(m-1) \text{ and } w_i> b\}$. Note that $c'$ is not necessarily unique. Thus, in order to create exactly one copy of $112$, we insert $b$ to the left of the right-most copy of $c'$. Next, let $R_{n,k,m}(112) = \{w\in R_{n,k}(112)\text{ with max } m\}$ and we have

$$\#R_{n,1}(112) = \sum_{m=2}^{n-1}(m-1)\#R_{n-1,0,m}(112).$$

All that remains is to compute $\#R_{n-1,0,m}(112)$. However, this is easily done as the characterization that Sagan provides in \cite{sag:pas} for $R_{n,0}(112)$ allows us to see that the problem reduces to picking a multisubset of $n-m-1$ elements from a set of $m$ elements. We now apply the multiset formula which says that the number of multisubsets of size $k$ of an $n$ element set is $\binom{n+k-1}{k}$. Therefore we have

$$\sum_{m=2}^{n-1}(m-1)\#R_{n-1,0,m}(112) = \sum_{m=2}^{n-1}(m-1)\binom{n-2}{n-m-1}.$$
After applying some algebra and well known binomial identities, we arrive at
$$\#R_{n,1}(112) = (n-2)2^{n-3},$$
which completes the proof.
\end{proof}
Next, we give a characterization of $R_{n,1}(122)$ and determine its cardinality.
\begin{lemma}\label{122lemma}
We have that
$$R_{n,1}(122) = \{ w\in R_n : w = 12v, \text{ where $2\in v$ and if $a>1$ then $\#a(v) \leq 1$}\}.$$
\end{lemma}
\begin{proof}
Suppose $w\in R_{n,1}(122)$. Then there must exist some subword $abc$ representing the copy of 122 in $w$. We can see that $b=c=2$ because otherwise both the first 1 and first 2 in the word would serve as representatives for $a$ and $w$ would contain two copies. We can also see that $w$ must start $12$ for similar reasons. It is also clear that repeating another element larger than 1 would introduce a second copy of 122. Thus we have shown that the LHS is contained in the RHS.

For the other direction, let $w$ be a member of the set on the RHS. By construction, there is a copy of 122 in $w$. Suppose, seeking a contradiction, that $abc$ was another copy. Because $w$ begins $12$, we can see that $b \neq 2$ because there is only one element to the left of the first 2 and the second 2 has no 2 to its left. Thus $b>2$. However, elements greater than $2$ are distinct in $w$ and thus $b\neq c$ which is a contradiction. \end{proof}

\begin{theorem}\label{122card}
We have that 
$$\# R_{n,1}(122) = (n-2)2^{n-3}.$$
\end{theorem}
\begin{proof}
Let $w\in R_{n,1}(122)$. Then by Lemma \ref{122lemma} it must be the case that $w$ is of the form $12w'$, where $2\in w'$ and every element greater than 2 appears in $w'$ at most once. Because $w$ is an RGF, this implies that $w_i \in \{1,2,\max\{w_1,\dots, w_{i-1}\}+1\}$. In $w'$, there are $n-2$ indices in which the 2 could appear. In the remaining $n-3$ indices, we can either choose to place a 1, or place $\max\{w_1,\dots, w_{i-1}\}+1$. 
\end{proof}

Recall that two patterns $v$ and $w$ are $k$-Wilf equivalent if $R_{n,k}(v) = R_{n,k}(w)$ for all $n\geq 0$. Sagan showed that $\#R_{n,0}(122) = \#R_{n,0}(112) = \#R_{n,0}(121) = \# R_{n,0}(123) = 2^{n-1}$ and thus these patterns are all $0$-Wilf equivalent. We immediately obtain from Theorems~\ref{112card} and \ref{122card} the following corollary. 
\begin{corollary}
The patterns 112 and 122 are $1$-Wilf equivalent.
\end{corollary}
Through numerical computations we find that of the above patterns, 112 and 122 are the only two that are $1$-Wilf equivalent and the two are not $2$-Wilf equivalent. This suggests that $k$-Wilf equivalence becomes a stronger property as $k$ increases.

\section{Connections to Integer Partitions: 112 and 122}\label{intpartsec}

In this section, we first give basic definitions and introduce the important concepts of combinatorial statistics, generating functions and $q$-analogs. We then define and characterize the generating functions $\LS_{n,1}(112)$ and $\RB_{n,1}(122)$ and show their relation to integer partitions with distinct parts. Lastly, we characterize $\LB_{n,1}(122)$ and show its relation to integer partitions with restrictions on length and part size.

\subsection{Combinatorial Statistics, Generating Functions and $q$-analogs}
A combinatorial statistic on a set $A$ is a map $st: A\rightarrow \N$.
We now define the four combinatorial statistics of Wachs and White \cite{ww:pqs}: $\lb, \ls, \rb, \rs$, standing for ``left bigger," ``left smaller," ``right bigger," and ``right smaller," respectively. We will give the definition for $\lb$ from which the definition of the other three statistics should be clear. Given an RGF $w=w_1\dots w_n$ we define $\lb(w_i)$ as the number of \emph{unique} elements to its left that are bigger than it. In other words $\lb(w_i)= \#\{ w_j: i > j \text{ and } w_i < w_j\}$. Then we define
$$\lb(w) = \sum_{i= 1}^n \lb(w_i).$$ For example $$\lb(12231421) = 0+0+0+0+2+0+2+ 3 = 7.$$

A \emph{$q$-analog} of a theorem or combinatorial object is a theorem or combinatorial object in the formal variable $q$ for which the original theorem or object is recovered when we take the limit as $q\rightarrow 1$. For example, we can define $$[n]_q = 1+q+q^2+q^3+\dots+q^{n-1}$$ to be a $q$-analog of the number $n$. Usually one studies $q$-analogs that arise naturally from combinatorial statistics. The most well-known $q$-analog is given by the Binomial Theorem. It says that the binomial coefficients are related to the following $q$-analog of $2^n$
$$(1+q)^n = \sum_{k=0}^n \binom{n}{k}q^k,$$
which arises from considering the combinatorial statistic $\#S$ on subsets $S\subseteq [n]$. 

In 1988, Wachs and White \cite{ww:pqs} studied the distribution of the above statistics on $RG(n,k)$, the set of RGFs of length $n$ and maximum $k$. They show that the statistics lead to $q$-analogs of the Stirling numbers of the second kind
$$\sum_{w\in RG(n,k)} q^{\lb(w)}p^{\ls(w)} = \sum_{w\in RG(n,k)} q^{\rs(w)}p^{\rb(w)} = S_{p,q}(n,k),$$
where the polynomials $S_{p,q}(n,k)$ follow a two variable $q$-analog of the usual recursion of the Stirling numbers of the second kind,
$$S_{p,q}(n,k) =\begin{cases}
p^{k-1}S_{p,q}(n-1,k-1)+[k]_{p,q}S_{p,q}(n-1,k) &\text{if }0\leq k\leq n\\
1 &\text{if }n=k=0;\\
0 &\text{otherwise,}
\end{cases}$$
where $[k]_{p,q} = p^{k-1}+p^{k-2}q+p^{k-3}q^{2}+\dots + pq^{k-2}+q^{k-1}$ is a two variable $q$-analog of $k$. This symmetry of this result is surprising, given that RGFs are asymmetric.

The single variable $S_{p,1}(n,k)$, introduced in 1961 by Gould \cite{go:qstir},  are interesting in their own right, having many combinatorial interpretations. In addition, $S_{p,1}(n,k)$ are the moments of the $q$-Charlier polynomials, a family of orthogonal polynomials. In 1995, de M\'edicis, et. al. \cite{dmsw:qc} gave a combinatorial proof of this fact utilizing the $\rs$ statistic.

Combining these statistics with pattern containment, we construct generating functions $$\LB_{n,k}(v) = \sum_{w\in R_{n,k}(v)} q^{\lb(w)},$$
which will be the primary objects of study in this paper. The key idea is that the important information is contained in the coefficients and exponents of the polynomial as the coefficient of $q^t$ will be $\#\{w\in R_{n,k}(v) : \lb(w) = t\}$. Campbell, et. al. \cite{ddggprs:rgf} and Goyt and Sagan \cite{gs:sps} studied these generating functions in the case when $k=0$. In \cite{man:set}, Mansour suggests the following research direction: What happens when $k>0$? This article seeks to investigate this question.

\subsection{Partitions with distinct parts}\label{sectdistparts}
The remainder of this section will be devoted to characterizing three of the generating functions for patterns 112 and 122. It will turn out that these are intimately related to integer partitions and so we will first give some preliminary definitions of integer partitions.

An integer partition $\lambda$ of a natural number $n$ is a weakly decreasing sequence of positive integers, called \emph{parts}, $(\lambda_1,\lambda_2,\dots,\lambda_t)$ such that 
$$|\lambda| :=\sum_{i=1}^t \lambda_i = n.$$ 

If $\lambda$ partitions $n$, we denote this as $\lambda\vdash n$. The \emph{length} of a partition $\ell(\lambda)$ is the number of parts, i.e. $\ell(\lambda) = t$. Though distinct from set partitions, introduced earlier, we will usually refer to integer partitions as partitions and the meaning should be clear from context.

Integer partitions are ubiquitous combinatorial objects, as they are such a simple concept but studying them often leads to deep and difficult questions. Most famously, an exact formula for the number of partitions of $n$ is not known.

We will be concerned with integer partitions with distinct parts, $\lambda_i\neq \lambda_j$ for $i\neq j$. We let $D_n$ be the set of partitions into distinct parts where each part is at most $n$. The generating function for $D_n$ is well known to be
$$\sum_{\lambda\in D_n} q^{|\lambda|} = \prod_{i=1}^n(1+q^i).$$
For completeness, we include a short combinatorial proof of this fact. To construct a partition with distinct parts $\lambda$, we either include the number $i$ as a part, or we don't. If we do include it, it contributes $i$ to $|\lambda|$, and contributes nothing if we don't. An integer partition is completely determined by its parts, so we count each $\lambda\in D_n$ exactly once on the right hand side of the equation.

Goyt and Sagan showed the following connection between RGFs and $D_n$.
\begin{proposition}[\cite{gs:sps}]
$$\LS_{n,0}(121)=\RB_{n,0}(121)=\prod_{i=1}^{n-1}(1+q^i).$$
\end{proposition}
Later, Campell, et. al. extended the work of Goyt and Sagan and proved the following
\begin{theorem}[\cite{ddggprs:rgf}]\label{distinct}
$$\LS_{n,0}(121)=\RB_{n,0}(121)=\prod_{i=1}^{n-1}(1+q^i) = \LS_{n,0}(112) = \RB_{n,0}(122).$$
\end{theorem}

It is interesting to examine what happens to the above equalities when we move from avoiding the patterns to containing exactly one copy. As $\#R_{n,1}(121) \neq \#R_{n,1}(112)$, it cannot be the case that we find a complete analogue of Theorem \ref{distinct}, however we do retain a partial version of the theorem.

\begin{theorem}\label{112thm}
We have 
$$\RB_{n,1}(122) = \LS_{n,1}(112)  =\sum_{\lambda \in D_{n-2}} q^{|\lambda|}[\ell(\lambda)]_q =  q[n-2]_q\prod_{i=2}^{n-2}(1+q^i).$$
\end{theorem}
We break the proof of Theorem \ref{112thm} down into several parts. 
First, we give an explicit bijection $\varphi: R_{n,0}(112)\rightarrow D_{n-1}$ by composing bijections given in \cite{gs:sps} and \cite{ddggprs:rgf} and include a proof for completeness. The map $\varphi$ will have the property that $\ls(w) = |\varphi(w)|$.

We will need some notation first. Given a word $w$ and an element $j$ of $w$, we define $\#j(w)$ to be the number of copies of $j$ in $w$. For example, we have $\#3(1213433) = 3$ and $\#2(122232)=4$.
\begin{lemma}[\cite{gs:sps,ddggprs:rgf}]\label{lemma112} There exists an explicit bijection
$$\varphi : R_{n,0}(112)\rightarrow D_{n-1}$$
such that for $w\in R_{n,0}(112)$ we have $\ls(w) = |\varphi(w)|$.
\end{lemma}
\begin{proof}
Let $w\in R_{n,0}(112)$ with maximum $m$. Replace all repeated elements $w_i$ in $w$ with $m-w_i+1$ to form a new word $w'$. Then $$\varphi(w):= (\lambda_{1},\lambda_{2},\dots,\lambda_{m-1})$$ where 
$$\lambda_{m-i} := \sum_{j=1}^{i}\#j(w').$$ For example, if we run the algorithm for constructing $\varphi$ on $12345665332\in R_{11,0}(112)$ we get the following: \begin{align*}
w &= 12345665332, \\
w' &= 12345612445, \\
\varphi(w) = \lambda &= (10,8,5,4,2).
\end{align*}
First, we must show that $\varphi$ is well-defined. Observe that $$\sum_{j=1}^{i}\#j(w')$$ is strictly increasing, as the growth property of RGFs assures that all values between $1$ and $m$ occur in $w'$. Therefore the parts of $\lambda$ are distinct and decreasing and so $\lambda$ is a partition with distinct parts. By construction $\lambda_1 < n-1$, as it is the count of all elements strictly less than $m$ in $w'$. Therefore $\lambda \in D_{n-1}$.

Next, we show that $\varphi$ is invertible. Given $\lambda\in D_{n-1}$, we must construct $\varphi^{-1}(\lambda) = w \in R_{n,0}(112)$. First, we construct $w'$ by getting the counts $\#i(w') = \lambda_{m-i} - \lambda_{m-i+1}$. We then replace each repeat $w'_i$ in $w'$ with $m-w'_i+1$ to obtain the elements of $w$. By the characterization of $R_{n,0}(112)$ given in Lemma \ref{sagan112}, we now observe that $w$ is completely determined by its elements; there is only one legal ordering of elements for $w$ to be in $R_{n,0}(112)$. We are essentially performing our algorithm for constructing $\lambda$ in reverse, and so we leave the reader to work out the details of showing that $\varphi^{-1}$ is really the inverse of $\varphi$.

Lastly, we must show that $\ls(w) = |\varphi(w)|$. First, observe that $\ls(w_\ell) = w_\ell-1$ by the growth property of RGFs. Therefore we can write 
$$\ls(w) = \sum_{\ell=1}^n (w_\ell -1) = \sum_{k=1}^m (k-1)\#k(w).$$ Next, observe that because we replace every repeat of $w_\ell$ with $m-w_\ell+1$ in $w'$ that 
$\# k (w) = \#(m-k+1)(w').$ Therefore
$$\sum_{k=1}^m (k-1)\#k(w) = \sum_{j=1}^m (m-j)\#j(w') = \sum_{j=1}^{m-1} \sum_{i=1}^{m-j}\#j(w').$$
We interchange orders of summation and use the definition of $\lambda$ to obtain
$$\sum_{i=1}^{m-1} \sum_{j=1}^{i}\#j(w') = \sum_{i=1}^{m-1} \lambda_{m-i} = |\lambda|,$$
which completes the proof.
\end{proof}
We will use $\varphi$ to construct our next bijection.
\begin{proposition}\label{ls112prop}
We have 
$$\LS_{n,1}(112) = \sum_{\lambda \in D_{n-2}} q^{|\lambda|}[\ell(\lambda)]_q = q[n-2]_q\prod_{i=2}^{n-2}(1+q^i).$$

\end{proposition}
\begin{proof}

First we define 
$$A_{n} = \{ (\lambda, b-1) : \lambda \in D_{n} \text{ and } b\in[\ell(\lambda)]\}$$ and note that 
$$\sum_{(\lambda,b-1)\in A_{n-2}} q^{|\lambda|+b-1} = \sum_{\lambda\in D_{n-2}} q^{|\lambda|} [\ell(\lambda)]_q.$$

We define a map $\psi: R_{n,1}(112)\rightarrow A_{n-2}$ as follows. Let $w\in R_{n,1}(112)$ and let $abc$ be the copy of $112$ in $w$. Delete the element $b$ from $w$ to obtain $w'\in R_{n-1,0}(112)$. Then we will construct $\psi(w) = (\varphi(w'), b-1)$. Continuing our example from Lemma~\ref{lemma112} with $abc$ in bold,  we have 
\begin{align*}
w &= 1\bm{2}3456653\bm{23}2,\\
w' &= 12345665332,\\
(\lambda, b-1) &= ((10,8,5,4,2),1).
\end{align*}

To see that $\psi$ is well defined, first note that because $abc$ is the only copy of $112$ in $w$, deleting $b$ ensures that $w'\in R_{n-1,0}(112)$. By Lemma~\ref{lemma112}, $\lambda \in D_{n-2}$ and $\ell(\lambda) = m-1$, where $m$ is the maximum of $w$. Therefore because $1\leq b\leq m-1$ we have $b\in[\ell(\lambda)]$.

We now construct $\psi^{-1}$. First, we need some notation. For $x\in \N$, let 
$$S(x,w) := \max_{i}\{ w_i : w_i> x\}.$$ We now begin to construct $\psi^{-1}(\lambda,b-1)$ for $(\lambda,b-1)\in A_{n-2}$. Let $w' = \varphi^{-1}(\lambda)$. And let $j = S(b,w')$. Insert $b$ at index $j-1$ in $w'$ to obtain $\psi^{-1}(\lambda,b)$. To see that $\psi^{-1}$ is well-defined, first note that $\varphi^{-1}(\lambda) \in R_{n-1,0}(112)$ by Lemma~\ref{lemma112}. By inserting $b$ to position $j-1$, we ensure that exactly one element of $w'$ is to the right of $b$ and bigger than it. Thus we create exactly one copy $abc$ of 112, where $a$ is the copy of $b$ in the initial run and $c$ is the element at index $j$. Note that there is precisely one way to insert $b$ into $w'$ and obtain exactly one copy of $112$ and thus it should be clear that these are indeed inverses.

Now we wish to show that if $\psi(w) = (\lambda, b-1)$ then $\ls(w) = |\lambda |+ b-1$. The element $b$ in $w$ does not change the value of $\ls(w')$ and thus $\ls(w') =|\lambda|$. In addition $\ls(b) = b-1$ and so we have established that

$$\LS_{n,1}(112) =\sum_{\lambda \in D_{n-2}} q^{|\lambda|}[\ell(\lambda)]_q.$$

We now proceed with a standard generating function argument and let $$f_n:= \sum_{\lambda \in D_{n-2}}q^{|\lambda|}[\ell(\lambda)]_q.$$ By telescoping each $[\ell(\lambda)]_q$ we can see that 
$$(1-q)f_n = \prod_{i=1}^{n-2}(1+q^i) - \prod_{i=1}^{n-2}(1+q^{i+1})$$
and after some algebra, we arrive at 
$$f_n = q[n-2]_q\prod_{i=2}^{n-2}(1+q^i),$$
which proves the proposition.
\end{proof}
Next we will show bijectively that 
\begin{proposition}\label{rb122prop}
We have 
$$\RB_{n,1}(122) = \LS_{n,1}(112).$$
\end{proposition}
\begin{proof}
We will construct a bijection $\eta: R_{n,1}(112) \rightarrow R_{n,1}(122)$ such that $\ls(w) = \rb(\eta(w))$ for $w=w_1\dots w_n \in R_{n,1}(112)$. Recall that by Lemma~\ref{lem112card}, $w = 123\dots m w'$, where $w'$ has one noninversion. Let $abc$ denote the copy of 112 in $w$ and let $m$ be the maximum of $w$. Begin with the initial run $12\dots m$. Place a 2 to the right of $m-b+1$ in the initial run. Denote the word we have created so far as $v = v_1\dots v_{m+1}$. Then, for each $w_i$, where $i>n$ and $w_i\neq b$, place a 1 to the right of $v_{m+2-w_i}$. We will let $\eta(w)$ be the resulting word.

Continuing our running example,
\begin{align*}
w &= 123456653232,\\
v &= 1232456,\\
\eta(w) &= 121312411516.
\end{align*}

To see that $\eta$ is well-defined, first observe that because $w$ contains the element 2, $m\geq 2$. Therefore $\eta(w)$ contains exactly two 2s and thus $\eta(w)$ contains a copy of 122. In addition, note that because $1\leq w_i \leq m$ we have $2\leq m+2-w_i \leq m+1$. Thus in the final step of the construction of $\eta(w)$ we only place 1s to the right of indices $i> 1$ and so $\eta(w)$ begins 12. Therefore, $\eta(w)$ contains exactly one copy of 122. 

To see that $\eta$ is invertible, first note that we can easily recover the value of every element of $w$. The elements greater than 2 are not effected by the bijection and we can obtain the value of the corresponding element of a 1 or the second 2 by counting the number of elements greater than 1 to the right of the 1 or 2. Lastly, note that once we have found the values of every element in $w$, there is exactly one way to order them to guarantee that $w$ contains 112 exactly once.

We now show that $\ls(w) = \rb(\eta(w))$. We will pair each element $w_i$ with an element $\eta(w_i)$ such that $\ls(w_i) = \rb(\eta(w_i))$. Recall that in any RGF $v$, we have $\ls(v_i) = v_i-1$. First note that by construction, $\max(w) = \max (\eta(w)) = m$. Thus we can pair each element $w_i$ of the initial run of $w$ with the first occurrence of $m-w_i +1$ in $\eta(w)$. For each of these elements we have $\ls(w_i) = \rb(\eta(w_i)) = w_i -1$. We can pair $b$ with the second 2 in $\eta(w)$. Because we place the $2$ to the right of $m-b+1$, there are exactly $b-1$ elements to the right of and bigger than the $2$. Similarly, for any $w_i \in w'$ where $w_i\neq b$, we pair up $w_i$ with the 1 to the right of $m-w_i+1$ and for such $w_i$ we have $\rb(\eta(w_i)) = w_i -1$. Therefore we have $\ls(w) = \rb(\eta(w))$.

%For each element $w_i$ of $w$, let $\eta(w_i)$ be the element of $\eta(w)$ that was created because of $w_i$. We now show that $\ls(w_i) = \rb(\eta(w_i))$ for all $i$ and thus $\ls(w) = \rb(\eta(w))$. Because the maximum of 
\end{proof}
Combining Propositions \ref{ls112prop} and \ref{rb122prop} we obtain Theorem \ref{112thm}.

\begin{corollary} The total number of parts in all partitions with distinct parts that are at most $n-2$ is $(n-2)2^{n-3}$.
\end{corollary}
\begin{proof} If we let $q=1$ in the above theorem then $\LS_{n,1}(112)$ becomes $\#R_{n,1}(112) = (n-2)2^{n-3}$. However Theorem~\ref{112thm} shows that this is equivalent to 
$$\sum_{\lambda \in D_{n-2} }1^{|\lambda|}[\ell(\lambda)]_1 = \sum_{\lambda\in D_{n-2}} \ell(\lambda).$$
\end{proof}

\subsection{Partitions with length and part size restrictions}
In this subsection, we examine generating functions related to integer partitions with length and part size restrictions.
Recall our introductory $q$-analog example
$$[n]_q := 1+q+q^2+\dots+q^{n-1}.$$ Using this, we can construct a $q$-analog of the factorial function
$$[n]_q! := [n]_q[n-1]_q[n-2]_q\dots[2]_q[1]_q.$$ The fun doesn't stop there, however. Some of the most well-known $q$-analogs are the \emph{Gaussian Binomial Coefficients}
$$\gauss{n}{k}_q := \frac{[n]_q!}{[k]_q![n-k]_q!}.$$ There are many combinatorial interpretations for these polynomials. We will focus on the following well known interpretation.
\begin{proposition}
The generating function for integer partitions $\lambda$ with $\ell(\lambda) \leq s$ and $\lambda_1 \leq t$ is given by
$$\gauss{s+t}{s}_q,$$ where the coefficient of $q^k$ is the number of such partitions of $k$.
\end{proposition}
One immediate consequence of this proposition is the surprising fact that the coefficients of the Gaussian polynomials are always integers even though polynomial division is involved.

Campbell et. al. proved the following connection between these polynomials and the generating functions we've been studying.
\begin{theorem}[\cite{ddggprs:rgf}]\label{122prev}
We have 
$$\LS_{n,0}(122) = \sum_{t\geq 0} \gauss{n-1}{t}_q.$$
\end{theorem}

As we move from $k=0$ to $k=1$, we find a similar phenomenon as in $\LS_{n,1}(112)$ occurs. We get a variation of Theorem~\ref{122prev} where the new generating function is related to the individual parts of the partitions counted by $\LS_{n,0}(122)$.
\begin{theorem}\label{lb122} We have that 

$$\LB_{n,1}(122) = \sum_{k\geq0} \#B_{n,k+1}q^{k},$$
 
 where $B_{n,k} = \{(\lambda, i) : \lambda = (\lambda_1,\dots, \lambda_m) \vdash k, 1\leq i \leq m, \text{ and } \ell(\lambda)+\lambda_1+1\leq n\}.$
 
 % \tre{Is there a closed form GF for this? Number of parts in all integer partitions that fit in some $t\times s$ box where $t+s \leq n-1$}
\end{theorem}
\begin{proof} We will construct a bijection $\tau: R_{n,1}(122)\rightarrow \bigcup_{k\geq0} B_{n,k+1}$ and show that $\tau(w) \in B_{n,k+1}$ if and only if $\lb(w) = k$.

 Let $w\in R_{n,1}(122)$ and let $\tau(w) = (\lambda, i)$. We define $\tau$ in the following way. Consider all $w_j = 1$ or $2$, where $w_j$ is not the first occurrence of that element and let $a = a_1\dots a_m$ be the subword consisting of the $w_j$. For each such $w_j$, create the part $(\max(w_1\dots w_{j-1})-1)$ in $\lambda$. By the characterization of $R_{n,1}(122)$, $a$ consists of exactly one 2 and $(m-1)$ copies of 1. We then let $i$ correspond to the index of the 2 in $a$. For example, $\tau(12\bm{1}3\bm{21}4) = ( (2,2,1), 2),$ where the subword $a$ is in bold. We can see that by construction $\ell(\lambda) + \lambda_1 \leq \max(w)\leq n-1$ and thus $\tau$ is well-defined.

Let $(\lambda,i)\in \bigcup_{k\geq 0} B_{n,k+1}$. To construct $w = \tau^{-1} (\lambda, i)$, we begin by creating an initial run $1\dots (\lambda_1+1)$. For each part $\lambda_j$, we insert a 1 to the right of the element $\lambda_j+1$ in the initial run $w$. If $(\lambda_1+ \ell(n)) < (n-1)$, we append the sequence $(\lambda_1+2) (\lambda_1+3) \dots (n-\ell(\lambda))$ to the end of $w$. Lastly, we change the $(i+1)$st copy of 1 in $w$ to a 2. As $\tau^{-1}$ is essentially $\tau$ in reverse, we leave the reader to check that they are indeed inverses.

Examining the construction of $\tau(w)$, we can see that only the elements that make up the subword $a$ contribute to $\lb(w)$. If we let $\lambda_q$ be the part of $\lambda$ corresponding to $a_q$ then we can see that $\lb(a_q) = \lambda_q - a_q +1$. As there is only one 2 in $a$, it must be the case that $\lb(w) = |\lambda| -1$. 
\end{proof}

%\begin{corollary} If $k\leq n-3$, then the coefficient of $q^k$ in $\LB_{n,1}(122)$ is the number of parts in all partitions of $k+1$.
%\end{corollary}

\begin{corollary} The number of parts of all partitions that fit in a $t\times s$ box where $t+s\leq n-1$ is $(n-2)2^{n-3}$.
\end{corollary}
\begin{proof}
Letting $q=1$ in $\LB_{n,1}(122)$ gives $\# R_{n,1}(122) = (n-2)2^{n-3}$. However Theorem \ref{lb122} shows that this is equivalent to 
$$\sum_{k\geq 0} \# B_{n,k+1}.$$
\end{proof}
\begin{corollary}
The number of parts of all partitions that fit in a $t\times s $ box where $t+s\leq n-1$ is equal to the number of parts in all partitions with distinct parts at most $n-2$.
\end{corollary}

The above corollaries show one of the powers of using generating functions. In examining the finer structure of the generating function, it is possible to develop new unexpected insights into the objects being studied. 
\section{Sperner Posets and Unimodality}\label{unimsec}
\subsection{Symmetry and unimodality}
After obtaining the characterizations of generating functions, it is interesting to determine if they satisfy any nice properties. Two commonly studied properties are symmetry and unimodality.

A polynomial $p$ is called \emph{symmetric} if the $i$th coefficient of $p$ is equal to the $(\deg(p)-i)$th coefficient of $p$. This is not to be confused with another common definition of symmetric (multivariate) polynomials, which involves invariance under permutations of its arguments.

A finite sequence of integers $a_1,\dots, a_n$ is called \emph{unimodal} if there exists an index $m$ such that 
\begin{align*}
1.&\hspace{10pt} a_1\leq a_2\leq a_3\leq \dots \leq a_m\\
2.&\hspace{10pt} a_m\geq a_{m+1} \geq a_{m+2} \geq \dots \geq a_n.
\end{align*}
A polynomial is unimodal if its coefficients are unimodal.

Stanley \cite{sta:uni} gives the following proposition, which will be of use to us.
\begin{proposition}[\cite{sta:uni}, Proposition 1]\label{stansym}
If $A(q)$ and $B(q)$ are symmetric, unimodal polynomials with nonnegative coefficients, then so is $A(q)B(q)$.
\end{proposition}
If we examine his proof, we can drop the unimodality condition to get a statement purely about symmetry.
\begin{corollary}
If $A(q)$ and $B(q)$ are symmetric polynomials with nonnegative coefficients, then so is $A(q)B(q)$.
\end{corollary}

Using the characterization given in Theorem~\ref{112thm} we immediately obtain.
\begin{proposition}
$\LS_{n,1}(112)/q$ is symmetric.
\end{proposition}

Note, however, that we cannot apply Proposition~\ref{stansym} to prove unimodality because terms such as $1+q^2$ have internal zeros and thus are not unimodal. This leads us to the following conjecture.
\begin{conjecture}\label{conj112}
$\LS_{n,1}(112)$ is unimodal.
\end{conjecture}

\begin{center}
\begin{figure}[h]
\begin{center}
\begin{tabular}{|c|c|}\hline
$n$ & $\LS_{n,1}(112)$\\\hline
3 & $q$ \\\hline
4 & $q^4 + q^3 + q^2 + q$\\\hline
5 & $q^8 + q^7 + 2q^6 + 2q^5 + 2q^4 + 2q^3 + q^2 + q$\\\hline
6 & $q^{13} + q^{12} + 2q^{11} + 3q^{10} + 3q^9 + 4q^8 + 4q^7 + 4q^6 + 3q^5 +
3q^4 + 2q^3 + q^2 + q$\\\hline
\end{tabular}
\end{center}
\caption{A table of $\LS_{n,1}(112)$ for $n = 3,4,5,6$}\label{figunim}
\end{figure}
\end{center}
Figure~\ref{figunim} shows $\LS_{n,1}(112)$ for the first few $n$ and one can see that in each case it is unimodal. Indeed, we have numerically checked that the conjecture hold up to $n=800$. 

It would be ideal if we could find a combinatorial proof of this fact. This would  involve using a map $\psi$ from the set elements $w$ with $\ls(w) = k$ to the set of elements $v$ with $\ls(\psi(v)) = k+1$. If we showed that $\psi$ was an injection where we believe $k$ is increasing and a surjection where we believe $k$ is decreasing, then we have a proof of unimodality. However, unlike symmetry, unimodality is often very difficult to prove and even more difficult to prove combinatorially. Stanley \cite{sta:uni} gives a survey of a myriad of proof techniques for unimodality. Unfortunately, most techniques only apply in very specific circumstances and were not of use for Conjecture~\ref{conj112}.

The unimodality of $\LS_{n,1}(112)$ may not be surprising given the following well-known theorem due to Hughes \cite{hug:lie}.
\begin{theorem}\label{distpartunim} The polynomial $$\prod_{i=1}^n (1+q^i)$$ is symmetric and unimodal.
\end{theorem}However, what is remarkable about this theorem is that it remains an open problem to give a combinatorial proof of this fact! Hughes originally gave a proof that used the representation theory of Lie algebras. Since then, a number of different proofs have been given, ranging from using analysis \cite{odr:uni} to algebraic geometry \cite{sta:lef}.  In \cite{pro:sol} Proctor gave a proof using elementary linear algebra (though he notes that the linear operators he uses come from the representations of Lie algebras).

We will give a summary of Proctor's proof and in doing so, formulate a stronger form of Conjecture~\ref{conj112}.

%There are many techniques to show unimodality. Stanley gives a very thorough survey in \tre{ cite stanely survey}. Many of the techniques he gives involve showing that the polynomial in actually log-concave, a property which implies unimodality. A sequence $a_1,\dots, a_n$ is \emph{log-concave} if 
%$a_i^2\geq a_{i-1}a_{i+1}$ for $2\leq i\leq n-1$. Unfortunately, we can see in Figure~\ref{figunim} that $\LS_{n,1}(112)$ is not log-concave.

\subsection{Posets and the Sperner property}

%\begin{tikzpicture}
%  \node (min) at (0,0) {$(1),0$};
%  \node (a) at (0,1) {$(2),0$};
%  \node (b) at (-1,2) {$(2,1),0$};
%  \node (c) at (1,2) {$(3),0$};
%  \node (d) at (-2,3) {$(2,1),1$};
%  \node (e) at (0,3) {$(3,1),0$};
%  \node (f) at (-1,4) {$(3,1),1$};
%  \node (g) at (1,4) {$(3,2),0$};
%  \node (h) at (0,5) {$(3,2),1$};
%  \node (i) at (2,5) {$(3,2,1),0$};
%  \node (j) at (1,6) {$(3,2,1),1$};     
%  \node (max) at (1,7) {$(3,2,1),2$};
%  \draw (min) -- (a) -- (b) -- (d) -- (f) -- (h) -- (j) -- (max)
%  (a) -- (c)  -- (e) -- (g) -- (i) -- (j)
%  (b) -- (e) -- (f)
%  (g) -- (h);
%\end{tikzpicture}

To give the details of Proctor's proof, we must translate our problem into the language of posets. The following definitions and basic theorems about posets can be found in most introductory combinatorics texts such as \cite{sta:enum}. Specific results about $M(n)$ can be found in \cite{pro:sol} and results about $M^1(n)$ are new to the best of the author's knowledge.

A \emph{partially ordered set} or \emph{poset} is a set $P$ (in our case finite) together with a relation denoted $\leq$ that has the following three properties:
\begin{enumerate}
\item \emph{(Reflexivity)} For all $a\in P$, $a\leq a$.
\item \emph{(Antisymmetry)} For all $a,b\in P$, if $a\leq b$ and $b\leq a$ then $a = b$.
\item \emph{(Transitivity)}  For all $a,b,c\in P$, if $a\leq b$ and $b\leq c$ then $a\leq c$.
\end{enumerate}
If $a\leq b$ or $b\leq a$, we say that $a$ and $b$ are \emph{comparable}. If $a\leq b$ and $b\neq a$, we write $a<b$. 

We can introduce partial orders $\leq_0$ and $\leq_1$ on the elements of $D_n$ and of $A_n$ from Section~\ref{sectdistparts} respectively. For $\lambda, \lambda'\in D_n$ we define 
$$\lambda \leq_0 \lambda' :\iff \lambda_i\leq \lambda'_i\text{ for all }1\leq i \leq n,$$
permitting trailing parts equal to 0 so that the definition is well defined. This poset is traditionally denoted $M(n)$ and has been well studied.

We will define $\leq_1$ in a similar manner. For $(\lambda,b), (\lambda',b')\in A_n$ let 
$$(\lambda,b) \leq_1 (\lambda',b') :\iff \lambda_i\leq \lambda'_i\text{ for all }1\leq i \leq n,\text{ and } b\leq b'.$$
We will denote this poset $M^1(n)$.

To continue we need some more definitions about posets, which we will give in general for any poset.
For elements $p,p'\in P$, we say that $p$ \emph{covers} $q$ if $p'< p$ and there is no element $q$ such that $p'< q< p$. Note that a poset can also be defined in terms of its covering relations; transitivity ensures that they completely define the partial order.

 It is often useful to visualize a poset through its \emph{Hasse diagram}, where a line is drawn connecting $p$ and $p'$ with $p$ above $p'$ if $p$ covers $p'$. The Hasse diagrams of $M(3)$ and $M^1(3)$ are shown in Figure~\ref{hasse}.

We say that $\rho: P \rightarrow \N$ is a \emph{rank function} for a poset $P$ if two properties hold.
\begin{enumerate}
\item If $p'\leq p$ then $\rho(p')\leq \rho(p)$.
\item If $p$ covers $p'$ then $\rho(p) = \rho(p')+1$.
\end{enumerate}
If a poset admits a rank function, then that poset is called \emph{ranked} and the \emph{ranks} of that poset are the subsets $P_i = \{ p\in P: \rho(p) = i\}$. Not all posets are ranked, though the posets studied in this paper will be.

\begin{proposition}
The posets $M(n)$ and $M^1(n)$ are ranked with rank functions $\rho_0(\lambda) = |\lambda|$ and $\rho_1(\lambda,b) = |\lambda|+b-1$.
\end{proposition}

The \emph{rank polynomial} of a poset is the generating function
$$\sum_{p\in P} q^{\rho(p)}$$
and the coefficients of the polynomial are called the \emph{Whitney numbers} of the poset. Here is where we find our connection with unimodality.

\begin{proposition}
The rank polynomials of $M(n)$ and $M^1(n)$ are $\prod_{i=1}^n (1+q^i)$ and $\LS_{n-2,1}(112)$ respectively.
\end{proposition}
A poset is \emph{rank-unimodal} if its rank polynomial is unimodal and \emph{rank-symmetric} if its rank polynomial is symmetric.
Thus we can translate Conjecture~\ref{conj112} and Theorem~\ref{distpartunim} into the equivalent statements that $M^1(n)$ and $M(n)$ are rank-unimodal. Proctor proves this version of Theorem~\ref{distpartunim}. However, Proctor actually proves a stronger statement.

In any poset $P$, a \emph{chain} is a subset of $P$ in which every element is comparable. Conversely, an \emph{antichain} is a subset in which no two elements are comparable.
A ranked poset is called \emph{Sperner} if the size of the largest antichain is at most the size of the largest rank. The property is named after Emanual Sperner, of Sperner's theorem, which says that the poset of subsets of $[n]$ ordered by set inclusion satisfies this property. Proctor proves the following theorem, which is originally due to Stanley \cite{sta:lef}.
\begin{center}
\begin{figure}[h]
\hspace{50pt}
\begin{tikzpicture}
  \node (min) at (0,0) {$(1),0$};
  \node (a) at (0,1.5) {$(2),0$};
  \node (b) at (-1.5,3) {$(2,1),0$};
  \node (c) at (1.5,3) {$(3),0$};
  \node (d) at (-3,4.5) {$(2,1),1$};
  \node (e) at (0,4.5) {$(3,1),0$};
  \node (f) at (-1.5,6) {$(3,1),1$};
  \node (g) at (1.5,6) {$(3,2),0$};
  \node (h) at (0,7.5) {$(3,2),1$};
  \node (i) at (3,7.5) {$(3,2,1),0$};
  \node (j) at (1.5,9) {$(3,2,1),1$};     
  \node (max) at (1.5,10.5) {$(3,2,1),2$};
  \draw (min) -- (a) -- (b) -- (d) -- (f) -- (h) -- (j) -- (max)
  (a) -- (c)  -- (e) -- (g) -- (i) -- (j)
  (b) -- (e) -- (f)
  (g) -- (h);
\end{tikzpicture}
\hspace{20pt}
\begin{tikzpicture}
  \node (min) at (0,0) {$\emptyset$};
  \node (a) at (0,1.5) {$(1)$};
  \node (b) at (0, 3) 	{$(2)$};
  \node (c) at (-1.5,4.5) {$(2,1)$};
  \node (d) at (1.5,4.5) {$(3)$};
  \node (e) at (0,6) {$(3,1)$};
  \node (f) at (0,7.5) {$(3,2)$};   
  \node (max) at (0,9) {$(3,2,1)$};
  \draw (min) -- (a) -- (b) -- (c) -- (e) -- (f) -- (max)
  (b) -- (d) -- (e);
\end{tikzpicture}
\caption{The Hasse diagrams for $M^1(3)$ (left) and $M(3)$ (right)}\label{hasse}
\end{figure}
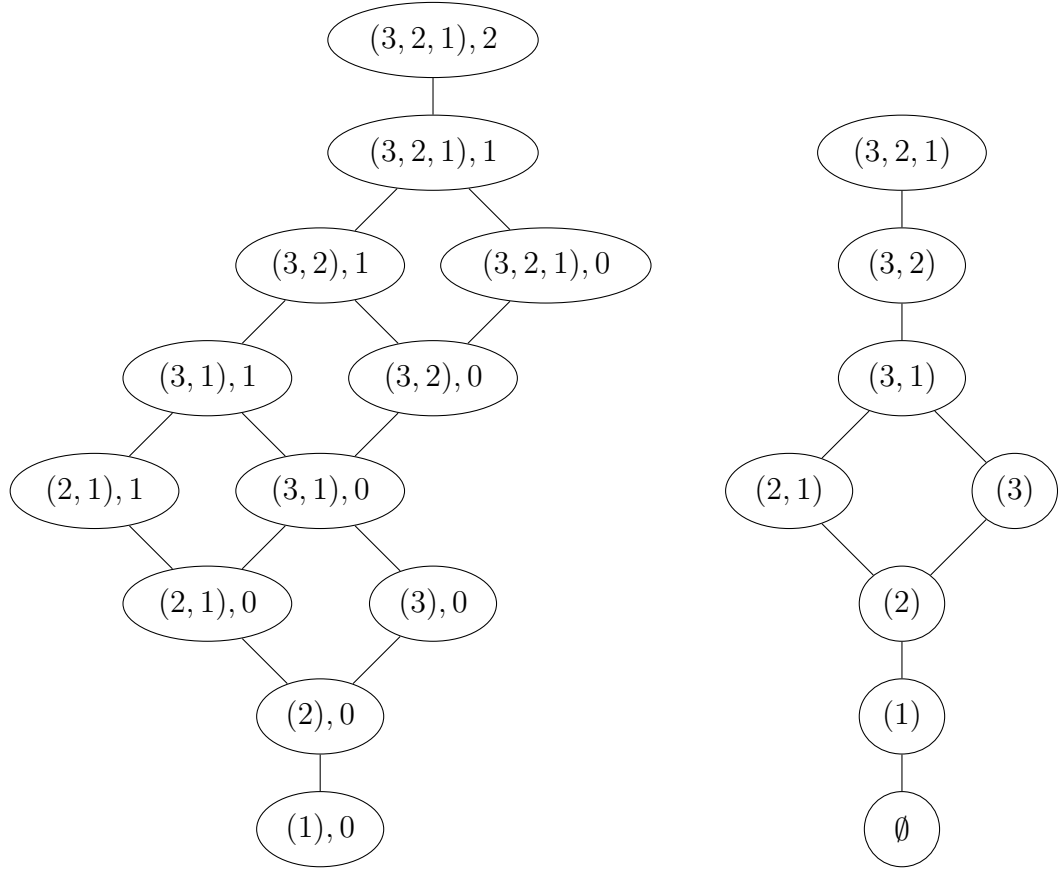
\end{center}

\begin{theorem}
The poset $M(n)$ is rank-symmetric, rank-unimodal, and Sperner.
\end{theorem}
We can now present a stronger form of Conjecture~\ref{conj112}.
\begin{conjecture}\label{strongconj}
The poset $M^1(n)$ is rank-symmetric, rank-unimodal, and Sperner.
\end{conjecture}
Numerically checking the Sperner property is much more difficult than checking unimodality and we have only been able to check that it holds up to $n=12$.

\subsection{Proctor's proof}
Though it does not appear to generalize to $M^1(n)$ due to the intrinsic link to representations of the Lie algebra sl(2,$\C$), we give a brief tour of Proctor's proof as it is the most elementary known proof of Theorem~\ref{distpartunim} as well as being interesting in its own right. Proctor proves that another poset is rank-unimodal and Sperner in addition to $M(n)$. Most of the details of the proof are worked out in this other case and so what appears after will be slightly different than what appears in Proctor as we focus on $M(n)$.

Denote the ranks of $M(n)$ as $M_0, M_1,\dots, M_{\binom{n+1}{2}}$. If we can find an injective map $f:M_i\rightarrow M_j$ such that for $\lambda\in M_i$, $\lambda<f(\lambda)$, then we will say that we have a \emph{matching} of $M_i$ into $M_j$. 
\begin{lemma} If we have a matching from $M_i$ into $M_{i+1}$ for all $0\leq i\leq \frac{1}{2}\binom{n+1}{2}$ and a matching from $M_{j+1}$ into $M_j$ for all $\frac{1}{2}\binom{n+1}{2}< j\leq \binom{n+1}{2}$ then $M(n)$ is rank-unimodal and Sperner.
\end{lemma}
\begin{proof}
A matching is an injection, which immediately gives rank-unimodality. Suppose we have an antichain $A\subset M(n)$. Our sequence of matching gives us a collection of chains, where each chain consists of elements that are successively matched with each other. Every element of $A$ must lie on some chain, and the number of chains is the size of the largest rank.
\end{proof}
We will nonconstructively prove the existence of these matchings using linear algebra. Suppose the elements of $M(n)$ are $\lambda_1,\lambda_2,\dots, \lambda_r$. Let $\til{M}$ be the vector space over $\C$ with basis elements $\til{\lambda_1},\til{\lambda_2},\dots,\til{\lambda_r}$. Denote $\til{M}_i$ as the subspace spanned by basis elements corresponding to elements of $M_i$.

We now introduce the \emph{order operator} of $M(n)$ as the linear operator $X$ on $\til{M}$
$$X\til{\lambda} = \sum_{\lambda' \emph{ covers } \lambda} \til{\lambda'}.$$
Note that $X(\til M_i) \subseteq \til M_{i+1}$. We let $X_i$ denote the restriction of $X$ to $\til M_i$. We can now translate our matching problem into linear algebra by the following lemma, stated without proof.
\begin{lemma}\label{proclem2} If there is some $h$ such that $X_i$ are injective for $i< h$ and surjective for $i\geq h$ then $M(n)$ is rank-unimodal and Sperner.
\end{lemma}
We will show that the criteria of the lemma hold for $h = \frac{1}{2}\binom{n+1}{2}$. To do this we will need to introduce two more linear operators $H$ and $Y$, chosen based upon the representation of sl(2,$\C$) associated to $M(n)$. Let 
$$H\til{\lambda} = \left(2|\lambda|- \binom{n+1}{2}\right)\til{\lambda}.$$ If $\lambda$ covers $\lambda'$ then it is easy to see that there is some index $i$ for which $\lambda_i = \lambda_i'+1$ and $\lambda_j = \lambda_j'$ for all $j\neq i$. Define
$$Y\til{\lambda} = \sum_{\lambda'\text{covered by } \lambda} c(\lambda,\lambda ')\til{\lambda'},$$
where $$c(\lambda,\lambda') = \begin{cases}
\binom{n+1}{2}, &\text{if }\lambda_i = 0\\
(n- \lambda_i)(n+\lambda_i + 1), &\text{otherwise.} 
\end{cases}$$
We now claim that the following commutation relations between the operators hold
\begin{align*}
HX-XH &= 2X,\\
HY-YH &=-2Y,\\
XY-YX &=H,
\end{align*}
omitting the proofs for the sake of brevity.
In the final part of the proof we will change bases. The construction of our new basis for $\til{M}$ will make clear why the conditions of Lemma~\ref{proclem2} hold.

Suppose we have a sequence of vectors $w_i, w_{i+1},\dots, w_s$ related by $w_{j+1} = Xw_j$, where $w_i\in \til{M}_i$. Then we will call this sequence a \emph{string} of vectors. Notice that if $w_j\in \til{M_j}$ then $w_{j+1}\in \til{M}_{j+1}$ by the definition of $X$. Our new basis will consist of strings of vectors symmetric about index $\frac{1}{2}\binom{n+1}{2}$.

Let $u_0 = \til\lambda$ be the first element of our basis, where $\lambda$ is the lone element of $M_0$. Let $U$ be the subspace of $\til{M}$ consisting of all linear combinations of $X$, $Y$, and $H$ applied to $u_0$. Repeatedly applying the commutation relations, we can express any element of $U$ as a linear combination of terms of the form $X^iH^jY^ku_0$. However $Yu_0 = 0$ and $Hu_0$ is a scalar multiple of $u_0$, so we can define a new string of vectors 
$$u_i = Xu_{i-1}$$
which span $U$. But $u_i\in \til{M}_i$, which are distinct disjoint subspaces of $\til M$, so the $u_i$ are linearly independent and thus a basis for $U$. Because $U$ is finite dimensional, we can call the last vector in the string $u_s$.

We now wish to determine the value of $s$. We can see that restricting $X$, $Y$, and $H$ to $U$ that we obtain operators on $U$, which we will denote $X'$, $Y'$ and $H'$. Now consider $u_k \in \til{M}_k$. Then
$$H'u_k = \left(2k - \binom{n+1}{2}\right)u_k$$
because $u_k$ is a linear combination of $\til{\lambda}$ where $|\lambda| = k$. Therefore $u_k$ is an eigenvector of $H'$ and we can compute the trace
$$\tr H' = \sum_{k=0}^s \left(2k-\binom{n+1}{2}\right).$$

We now come to the main reason we changed bases. For any linear operators $A,B$ we have 
$$\tr AB = \tr BA,$$
so we can apply this to $X'$ and $Y'$ to conclude that 
$$\tr H' = \tr(X'Y' - Y'X') = 0.$$
Combining this with our formula for $\tr H'$ we find that $s = \binom{n+1}{2}$.

We now continue with our construction by letting $\beta$ be the smallest index such that $\til M_\beta \not\subseteq U$ and let $v_\beta\in \til M_\beta\setminus U$. Let $V$ be the subspace resulting from letting $X,Y,H$ act on $v_\beta$ or $u_0$, which is spanned by the $u_i$ and $$v_\beta, \;\;v_{\beta+1} :=Xv_\beta, \;\;\dots.$$

By construction, if $v_q = u_q$ for some $q$ then $v_r = u_r$ for all $r\geq q$. Let $v_t$ be the largest element such that $v_t\neq u_t$. The only linear dependencies between the $u$ and $v$ can occur between elements in the same rank subspace, but $v_i \notin U$ for $\beta\leq i\leq r$, so the $u$ and $v$ must be linearly independent. We can use the same trace trick to find that the string of $v$'s is symmetric about $\frac{1}{2}\binom{n+1}{2}$, i.e. $t = \binom{n+1}{2} - \beta$. So the union of both strings is a basis for $V$. 

We can continue this process to form a sequence of subspaces
$$U\subseteq V \subseteq V_1 \subseteq V_2 \subseteq \dots \subseteq \til M,$$
which must terminate because $\til M$ is finite dimensional. Let $Z$ be the union of all strings created in this process. We can see that $Z$ forms a basis for $\til M$ and each string is symmetric around $\frac{1}{2}\binom{n+1}{2}$.

Now we show the criteria of Lemma~\ref{proclem2} hold. By construction, the $X_i$ map elements of $Z$ to elements in the same string. Because each string is symmetric around $\frac{1}{2}\binom{n+1}{2}$, it must be the case that $X_i$ is injective for $i< \frac{1}{2}\binom{n+1}{2}$ and surjective for $i\geq \frac{1}{2}\binom{n+1}{2}$. Thus, the critera of Lemma~\ref{proclem2} hold and $M(n)$ is rank-unimodal and Sperner. This completes the proof.

\subsection{Symmetric chain decompositions}
\begin{center}
\begin{figure}[h]
\begin{tikzpicture}
  \node (min) at (0,0) {$(1),0$};
  \node (a) at (0,1.5) {$(2),0$};
  \node (b) at (-1.5,3) {$(2,1),0$};
  \node (c) at (1.5,3) {$(3),0$};
  \node (d) at (-3,4.5) {$(2,1),1$};
  \node (e) at (0,4.5) {$(3,1),0$};
  \node (f) at (-1.5,6) {$(3,1),1$};
  \node (g) at (1.5,6) {$(3,2),0$};
  \node (h) at (0,7.5) {$(3,2),1$};
  \node (i) at (3,7.5) {$(3,2,1),0$};
  \node (j) at (1.5,9) {$(3,2,1),1$};     
  \node (max) at (1.5,10.5) {$(3,2,1),2$};
  \draw (min) -- (a) -- (b) -- (d) -- (f) -- (h) -- (j) -- (max)
  (a) -- (c)  -- (e) -- (g) -- (i) -- (j)
  (b) -- (e) -- (f)
  (g) -- (h);
\end{tikzpicture}
\hspace{20pt}
\begin{tikzpicture}
  \node (min) at (0,0) {$(1),0$};
  \node (a) at (0,1.5) {$(2),0$};
  \node (b) at (0,3) {$(2,1),0$};
  \node (c) at (3,3) {$(3),0$};
  \node (d) at (0,4.5) {$(2,1),1$};
  \node (e) at (3,4.5) {$(3,1),0$};
  \node (f) at (0,6) {$(3,1),1$};
  \node (g) at (3,6) {$(3,2),0$};
  \node (h) at (0,7.5) {$(3,2),1$};
  \node (i) at (3,7.5) {$(3,2,1),0$};
  \node (j) at (0,9) {$(3,2,1),1$};     
  \node (max) at (0,10.5) {$(3,2,1),2$};
  \draw (min) -- (a) -- (b) -- (d) -- (f) -- (h) -- (j) -- (max)
  (c) -- (e)  -- (g) -- (i);
\end{tikzpicture}
\caption{A symmetric chain decomposition (right) for the poset $M^1(3)$ (left).}\label{symchain}
\end{figure}
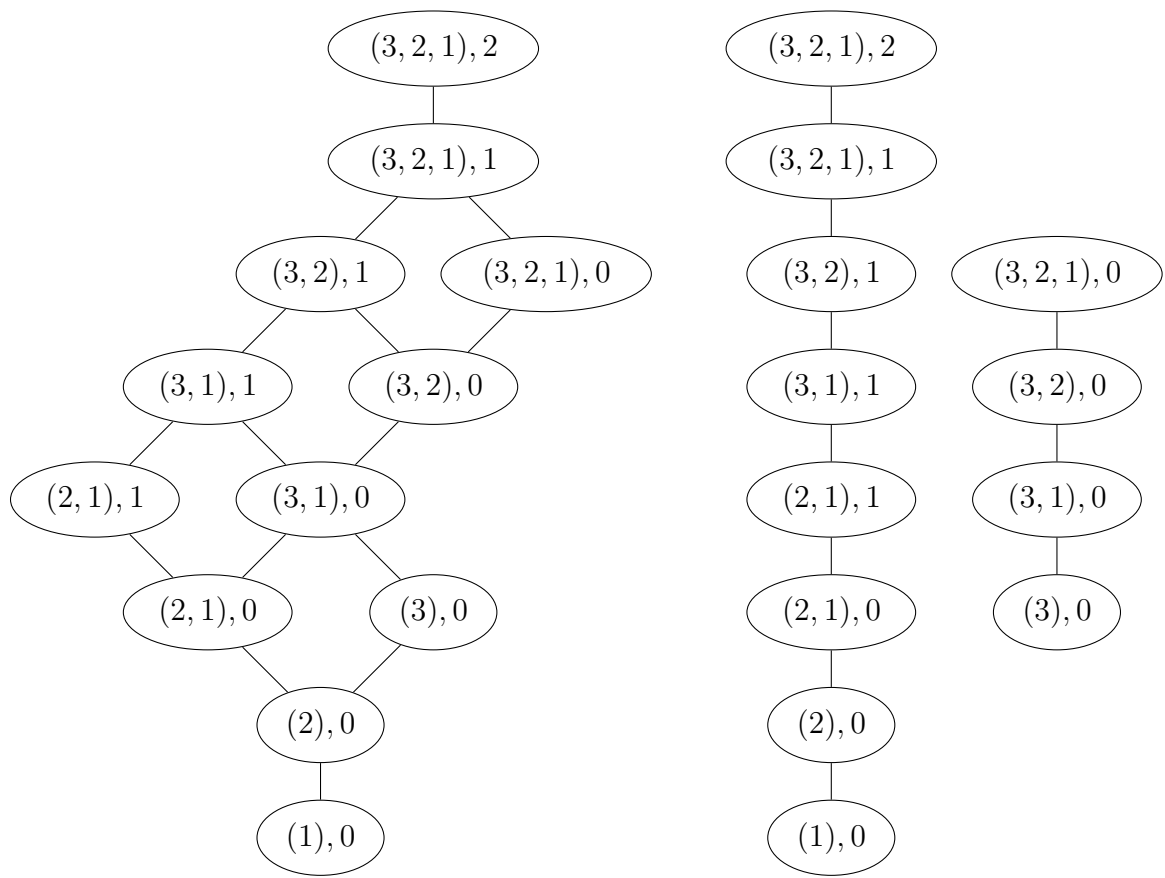
\end{center}
It would be pleasing to have a constructive, combinatorial proof of unimodality of either $M(n)$ or $M^1(n)$. One very nice proof technique is to construct what is called a \emph{symmetric chain decomposition}, which is a set of chains that cover the poset, each symmetric about the middle of the poset. Such a decomposition gives a symmetric  An example is given in Figure~\ref{symchain} which depicts a symmetric chain decomposition for $M^1(3)$. It is unknown whether or not a symmetric chain decomposition exists in general for $M(n)$ or $M^1(n)$. However, a symmetric chain decomposition has been found that gives a combinatorial proof of the following theorem.
\begin{theorem}
The polynomial $\gauss{n}{k}_q$ is unimodal.
\end{theorem}
This theorem was first stated by Cayley and then proven by Sylvester \cite{syl:gau} in 1877. It was not until over a hundred years later that a combinatorial proof was found. O'Hara \cite{oha:gau} gave an algorithm to construct a symmetric chain decomposition for a poset with the Gaussian polynomial as its rank polynomial. Perhaps someone will find a symmetric chain decomposition for $M(n)$ in the next hundred years.

\section{Open problems and further directions}
Conjectures \ref{conj112} and \ref{strongconj} are one of the main focuses of this article and finding a proof would be pleasing. However, there are several other further directions one could travel.

There are 5 patterns of length 3 and only a few have been mentioned in this article. One could try to give exact formulas for all of the generating functions for every pattern of length 3. Then one could move on to length 4 and so on.

Lastly, we introduced the concept of $k$-Wilf equivalence. Two natural questions arise.

\begin{question}
If two patterns $v$ and $w$ are $k$-Wilf equivalent, does this imply that they are also $(k-1)$-Wilf equivalent?
\end{question}
We have seen that this is the case in the examples given in this paper, though it is not obvious that it should hold in general, given that changing a word by just one element can change the number of copies of a pattern it contains by more than one. For example, $1112$ contains one copy of 111 but $11112$ contains four copies.
\begin{question}
Do there exist patterns $w$ and $v$ that are $k$-Wilf equivalent for all $k\geq 0$?
\end{question}
We have seen in this article an example of 1-Wilf equivalence and many examples of 0-Wilf equivalence appear in the literature. It does not seem out of the question that two patterns could be $k$-Wilf equivalent for all $k\geq 0$.

\section{Acknowledgments}
The author would like to thank both Prof. Kevin Woods, who supervised this thesis, and Prof. Bruce Sagan, who led the REU which started this project, for their mentorship and encouragement. The author would also like to thank the entire mathematics department at Oberlin, who created a great home for several years.

%\nocite{*}
\bibliographystyle{alpha}

\bibliography{Dorward_honors_thesis}

\end{document}